\documentclass{amsart}
\usepackage{verbatim}
\usepackage{amssymb}
\usepackage{amsmath}
\usepackage[usenames]{color}
\usepackage{graphicx}
\usepackage{amscd}
\usepackage[numbers,sort,square]{natbib}
\usepackage[T1]{fontenc}

\usepackage{enumerate}

\theoremstyle{plain}
\newtheorem{theorem}{Theorem}

\newtheorem{proposition}[theorem]{Proposition}

\newtheorem{corollary}[theorem]{Corollary}

\newtheorem{definition}[theorem]{Definition}

\theoremstyle{definition}

\newtheorem{remark}[theorem]{Remark}

\numberwithin{equation}{section}
\numberwithin{theorem}{section}
\allowdisplaybreaks

\usepackage[colorinlistoftodos,prependcaption,color=yellow,textsize=tiny]{todonotes}

\newcommand{\eps}{\varepsilon}

\newcommand{\ud}[0]{\,\mathrm{d}}

\newcommand{\vertiii}[1]{{\left\vert\kern-0.25ex\left\vert\kern-0.25ex\left\vert #1
    \right\vert\kern-0.25ex\right\vert\kern-0.25ex\right\vert}}

\begin{document}

\title[On strongly orthogonal martingales in UMD Banach spaces]
{On strongly orthogonal martingales\\ in UMD Banach spaces}

\author{Ivan Yaroslavtsev}
\address{Delft Institute of Applied Mathematics\\
Delft University of Technology \\ P.O. Box 5031\\ 2600 GA Delft\\The
Netherlands}
\email{I.S.Yaroslavtsev@tudelft.nl}

\begin{abstract}
In the present paper we introduce the notion of strongly orthogonal martingales. Moreover, we show that for any UMD Banach space $X$ and for any $X$-valued strongly orthogonal martingales $M$ and $N$ such that $N$ is weakly differentially subordinate to $M$ one has that for any $1<p<\infty$
\[
\mathbb E \|N_t\|^p \leq \chi_{p, X}^p \mathbb E \|M_t\|^p,\;\;\; t\geq 0,
\]
with the sharp constant $\chi_{p, X}$ being the norm of a
decoupling-type martingale transform and being within the range
\[
\max\Bigl\{\sqrt{\beta_{p, X}}, \sqrt{\hbar_{p,X}}\Bigr\} \leq \max\{\beta_{p, X}^{\gamma,+}, \beta_{p, X}^{\gamma, -}\} \leq \chi_{p, X} \leq \min\{\beta_{p, X}, \hbar_{p,X}\},
\]
where $\beta_{p, X}$ is the UMD$_p$ constant of $X$, $\hbar_{p, X}$ is the norm of the Hilbert transform on $L^p(\mathbb R; X)$, and $\beta_{p, X}^{\gamma,+}$ and $ \beta_{p, X}^{\gamma, -}$ are the Gaussian decoupling constants.
\end{abstract}



\maketitle

\section{Introduction}

Weak differential subordination of Banach space-valued martingales was recently discovered in the  papers \cite{Y17FourUMD,Y17MartDec,Y17UMD^A,OY18} as a natural extension of differential subordination in the sense of Burkholder and Wang (see \cite{Wang,Burk84}) to infinite dimensions, and it has the following form: for a given Banach space $X$ an $X$-valued martingale $N$ is {\em weakly differentially subordinate} to an $X$-valued local martingale $M$ if a.s.\ 
\[
|\langle N_0, x^*\rangle| \leq |\langle M_0, x^*\rangle| \;\; \text{and}
\] 
\[
[\langle N, x^*\rangle]_t - [\langle N, x^*\rangle]_s \leq [\langle M, x^*\rangle]_t - [\langle M, x^*\rangle]_s,\;\;\; 0\leq s\leq t,
\]
for any $x^*\in X^*$, where $[\,\cdot\,]$ is a {\em quadratic variation} of a martingale (see Subsection~\ref{subsec:qvandpdm}). 

Weak differential subordination, especially if $X$ satisfies {\em the UMD property} (see Subsection \ref{subsec:UMD}), has several applications in Harmonic Analysis. 
On the one hand, $L^p$-bounds for weakly differential subordinated {\em purely discontinuous} martingales imply estimates for $L^p$-norms of {\em L\'evy multipliers}. Namely, it was shown in  \cite{Y17FourUMD} that if $T_m$ is a L\'evy multiplier (i.e.\ a Fourier multiplier generated by a L\'evy measure, see \cite{BB,BBB}), then by using weakly differential subordinated purely discontinuous martingales one gets that for any $1<p<\infty$  the $L^p$-norm of $T_m$ acting on $X$-valued functions is bounded by {\em the UMD constant} $\beta_{p, X}$ (which boundedness characterizes the UMD property, please see Subsection \ref{subsec:UMD}).

On the other hand, various bounds for weakly differential subordinated {\em orthogonal} martingales coincide with the same type of estimates for the {\em Hilbert transform} (see \cite{OY18} by Os\c{e}kowski and the author). Recall that two $X$-valued martingales $M$ and $N$ are orthogonal if a.s.\ for any $x^* \in X^*$
\[
\langle M_0, x^*\rangle\cdot\langle N_0, x^*\rangle=0 \;\; \text{and}\;\;
[\langle M, x^*\rangle,\langle N, x^*\rangle]_t=0,\;\;\; t\geq 0,
\]
where $[ \;\cdot\;, \,\cdot\;]$ is a {\em covariation} of two martingales (see Subsection \ref{subsec:qvandpdm}).  
In particular, it was shown in \cite{OY18} that for any UMD Banach space $X$ and any $X$-valued orthogonal martingales $M$ and $N$ such that $N$ is weakly differentially subordinate to $M$ one has that for every $1<p<\infty$
\[
\mathbb E \|N_t\|^p \leq \hbar_{p, X}^p \mathbb E \|M_t\|^p,\;\;\; t\geq 0,
\]
where the sharp constant $\hbar_{p, X}$ is the norm of the Hilbert transform on $L^p(\mathbb R; X)$.

The goal of the present paper is to present sharp $L^p$ estimates for {\em strongly orthogonal} weakly differentially subordinated martingales. We call two $X$-valued martingales $M$ and $N$ strongly orthogonal if a.s.\ for any $x^*, y^* \in X^*$
\[
\langle M_0, x^*\rangle\cdot\langle N_0, y^*\rangle=0 \;\; \text{and}\;\;
[\langle M, x^*\rangle,\langle N, y^*\rangle]_t=0,\;\;\; t\geq 0.
\]
A classical example of strongly orthogonal martingales are stochastic integrals $\int \Phi\ud W$ and $\int \Phi\ud \widetilde W$, where $\Phi$ is $X$-valued elementary predictable, and $W$ and $\widetilde W$ are independent Brownian motions.
In the present paper we prove that for any strongly orthogonal weakly differentially subordinated martingales $M$ and $N$
\begin{equation}\label{eq:INTROSOWDSLpbounds}
\mathbb E \|N_t\|^p \leq \chi^p \mathbb E \|M_t\|^p,\;\;\; t\geq 0,\;\;1<p<\infty,
\end{equation}
where the sharp constant $\chi = \chi_{p, X}$ is within the range
\begin{equation}\label{eq:INTROchiestimates}
\max\{\sqrt{\beta_{p,X}}, \sqrt{\hbar_{p, X}}\} \leq \chi_{p, X} \leq \min\{\beta_{p, X}, \hbar_{p, X}\}.
\end{equation}

The main technique we used in order to prove \eqref{eq:INTROSOWDSLpbounds} is the {\em Bellman function method}. More specifically, we show that the following are equivalent
\begin{enumerate}[(A)]
\item \eqref{eq:INTROSOWDSLpbounds} holds for a constant $\chi>0$,
\item there exists $U^{SO}:X+iX \to \mathbb R$ such that $U^{SO}(x) \geq 0$ for any $x\in X$, $z \mapsto U^{SO}(x_0 + iy_0 + zx)$ in subharmonic in $z\in \mathbb C$ for any $x_0, y_0, x\in X$, and
\[
U^{SO}(x+iy) \leq \chi^p \|x\|^p - \|y\|^p,\;\;\;x,y \in X.
\]
\end{enumerate}
Notice that this method is not new while working with martingales with values in UMD Banach space. Namely, in \cite{Y17FourUMD} there was applied the {\em Burkholder function} $U:X \times X \to \mathbb R$ which first appeared in the paper \cite{Burk86} by Burkholder, and in \cite{OY18} there was used a {\em plurisubhirmonic function} $U_{\mathcal H}:X+iX\to \mathbb R$which first was constructed in the paper \cite{HKV03} by Hollenbeck, Kalton, and Verbitsky. The novelty of the present paper is in minimizing the necessary properties of the Bellman function. Namely, both $-U$ and $U_{\mathcal H}$ satisfy the property ${\rm (B)}$ outlined above (which makes the upper bound of \eqref{eq:INTROchiestimates} elementary).

In order to show the lower bounds of \eqref{eq:INTROchiestimates} and in order to characterize the least admissible cosntant $\chi_{p, X}$ we will need the example presented above. It turned out in Section \ref{sec:chipX} and \ref{sec:WDSofSOM} that the sharp constant $\chi_{p, X}$ is the smallest constant $\chi>0$ such that for any independent Brownian motions $W$ and $\widetilde W$ and for any elementary predictable $X$-valued $\Phi$ one has that
\[
\mathbb E \Bigl\|\int_0^{\infty}\Phi \ud \widetilde W \Bigr\|^p \leq \chi^p \mathbb E \Bigl\|\int_0^{\infty}\Phi \ud W \Bigr\|^p.
\]
Thus the desires lower bound of \eqref{eq:INTROchiestimates} follows from the well-known decoupling-type inequalities of Garling, see \cite{Gar85}.

\smallskip

Notice that if $X = \mathbb R$, then $\chi_{p, X} = \hbar_{p, X}$ (see Remark \ref{rem:-UandUHarediagplsh}). Nevertheless, it remains open whether this equality holds for a general UMD Banach space $X$. Moreover, if this is the case, then it proves a celebrated open problem about linear dependence of the constants $\beta_{p, X}$ and $\hbar_{p,X}$, see \cite[p.\ 48]{Bour84} and \cite{HNVW1,GM-SS,Y17FourUMD,OY18} (so far only a square dependence is known, see \eqref{eq:sqrtbetaleqhleqbeta^2}).

\medskip

\emph{Acknowledgment} --The author would like to thank Adam Os\c{e}kowski and Mark Veraar
for helpful comments. The author thanks Stefan Geiss for fruitful
discussions and for being the host while author's stay at
Jyv\"askyl\"a University where the present paper was written.

\section{Preliminaries}

Throughout the paper all Banach spaces are assumed to be over the scalar field $\mathbb R$ unless stated otherwise. We also assume that any filtration satisfies the usual conditions. In particular, any filtration is right-continuous, and thus all the local martingales exploited in the article have {\em c\`adl\`ag} versions (i.e.\ versions which are right continuous with left limits, see \cite{VerPhD,Y17FourUMD}). Furthermore, for any Banach space $X$, for any c\`adl\`ag process $A:\mathbb R_+ \times \Omega \to X$, and  for any stopping time $\tau$ we define
\[
 \Delta A_{\tau} := \lim_{\eps\to 0} (A_{\tau} - A_{(\tau-\eps)\vee 0}).
\]

\subsection{UMD Banach spaces}\label{subsec:UMD}

A Banach space $X$ is called {\it UMD} if for some (equivalently,
for all) $p \in (1,\infty)$ there exists a constant $\beta>0$ such
that for every $N \geq 1$, every martingale difference sequence
$(d_n)^N_{n=1}$ in $L^p(\Omega; X)$, and every $\{-1,1\}$-valued
sequence $(\varepsilon_n)^N_{n=1}$ we have
\begin{equation*}\label{eq:defofUMDconstant}
 \Bigl(\mathbb E \Bigl\| \sum^N_{n=1} \varepsilon_n d_n\Bigr\|^p\Bigr )^{\frac
1p}
\leq \beta \Bigl(\mathbb E \Bigl \| \sum^N_{n=1}d_n\Bigr\|^p\Bigr )^{\frac 1p}.
\end{equation*}
The least admissible constant $\beta$ is denoted by $\beta_{p,X}$
and is called the {\it UMD$_p$~constant} or, in the case if the
value of $p$ is understood, the {\em UMD constant} of $X$. It is
well-known that UMD spaces obtain a large number of useful
properties, such as being reflexive. Examples of UMD spaces
include all finite dimensional spaces and the reflexive range of
$L^q$-, Besov, Sobolev, Schatten class, and Musielak--Orlicz
spaces. Example of spaces without the UMD property include all
nonreflexive Banach spaces, e.g.\ $L^1(0,1)$ or $C([0,1])$. We
refer to \cite{Burk01,HNVW1,Rubio86,Pis16} for details.

\subsection{Quadratic variation}\label{subsec:qvandpdm}
Let $(\Omega, \mathcal F, \mathbb P)$ be a probability space with
a filtration $\mathbb F = (\mathcal F_t)_{t\geq 0}$ that satisfies
the usual conditions. Let $M:\mathbb R_+ \times \Omega \to \mathbb
R$ be a local martingale. We define a {\em quadratic variation} of
$M$ in the following way:
\begin{equation}\label{eq:defquadvar}
 [M]_t  := |M_0|^2 + \mathbb P-\lim_{{\rm mesh}\to 0}\sum_{n=1}^N \bigl|M(t_n)-M(t_{n-1})\bigr|^2,
\end{equation}
where the limit in probability is taken over partitions $0= t_0 <
\ldots < t_N = t$. Note that $[M]$ exists and is nondecreasing
a.s. The reader can find more on quadratic variations in
\cite{Kal,Prot,DM82}. For any martingales $M, N:\mathbb R_+ \times
\Omega \to \mathbb R$ we can define a {\em covariation}
$[M,N]:\mathbb R_+ \times \Omega \to \mathbb R$ as $[M,N] :=
\frac{1}{4}([M+N]-[M-N])$. Since $M$ and $N$ have c\`adl\`ag
versions, $[M,N]$ has a c\`adl\`ag version as well (see e.g.
\cite[Theorem I.4.47]{JS}).

\smallskip

A local martingale $M:\mathbb R_+ \times\Omega \to \mathbb R$ is
called {\em purely discontinuous} if $[M]$ is a.s.\ pure jump,
i.e.\ $[M]_t = \sum_{0\leq s \leq t}\Delta [M]_s$ a.s. Let $X$ be
a Banach space. Then an $X$-valued local martingale $M:\mathbb R_+
\times \Omega \to X$ is called {\em purely discontinuous} if
$\langle M, x^*\rangle$ is purely discontinuous for any $x^* \in
X^*$. Note that if $X$ is UMD, then any local martingale $M$ has a
unique decomposition into a sum of a continuous  local martingale
$M^c$ with $M^c_0=0$ and a purely discontinuous local martingale
$M^d$ (see \cite{Y17GMY}). We refer to
\cite{Kal,JS,Y17FourUMD,Y17MartDec,Y17GMY} for details on purely
discontinuous martingales.

\subsection{Weak differential subordination of martingales}

Let $X$ be a Banach space.
 Let $M,N:\mathbb R_+ \times \Omega \to X$ be local martingales. Then we say
that $N$ is {\em weakly differentially subordinate} to $M$ (we will denote this by $N \stackrel{w} \ll M$) if for each $x^* \in
X^*$ one has that $[\langle M,x^*\rangle]-[\langle N,x^*\rangle]$ is an a.s.\
nondecreasing function and $|\langle N_0,x^*\rangle|\leq |\langle
M_0,x^*\rangle|$~a.s.

The definition above first appeared in \cite{Y17FourUMD} as a natural
extension of differential subordination of real-valued
martingales. Later in \cite{Y17MartDec} there were obtained the
first $L^p$-estimated for weakly differentially subordinated
martingales, which have been significantly improved in \cite{OY18}
in the continuous-time case.

\subsection{Orthogonal martingales}

Let $M$ and $N$ be local martingales taking values in a given Banach
space $X$. Then $M$ and $N$ are said to be {\em orthogonal}, if
$\langle M_0, x^*\rangle\cdot \langle N_0, x^*\rangle  =0 $  and
$[\langle M, x^*\rangle,\langle N, x^*\rangle] = 0$ almost surely
for all functionals $x^* \in X^*$.

\begin{remark}\label{rem:orth+wds}
Assume that $M$ and $N$ are local martingales taking values in some
Banach space $X$. If $M$ and $N$ are orthogonal and $N$ is weakly
differentially subordinate to $M$, then $N_0=0$ almost surely
(which follows immediately from the above definitions, see \cite{OY18}).  Moreover,
under these assumptions, $N$ must have continuous trajectories
with probability $1$. Indeed, in such a case for any fixed $x^*\in
X^*$ the real-valued local martingales $\langle M, x^*\rangle$ and 
$\langle N, x^*\rangle$ are orthogonal and we have $\langle N,
x^*\rangle \ll \langle M, x^*\rangle$. Therefore, $\langle N,
x^*\rangle$ has a continuous version for each $x^*\in X^*$ by
\cite[Lemma 3.1]{Os09a} (see also \cite[Lemma 1]{BanWang96}),
which in turn implies that $N$ is continuous since any $X$-valued local
martingale has a c\`adl\`ag version.
\end{remark}

\subsection{Stochastic integration}
For given Banach spaces $X$ and $Y$, the symbol $\mathcal{L}(X,Y)$
will denote the classes of all linear operators from $X$ to $Y$.
We will also use the notation $\mathcal{L}(X)=\mathcal{L}(X,X)$.
Suppose that $H$ is a Hilbert space. For each $h\in H$ and $x\in
X$, we denote by $h\otimes x$ the associated linear operator given
by $g\mapsto \langle g, h\rangle x$, $g\in H$. The process $\Phi:
\mathbb R_+ \times \Omega \to \mathcal L(H,X)$ is called
\textit{elementary predictable} with respect to the filtration
$\mathbb F = (\mathcal F_t)_{t \geq 0}$ if it is of the form
\begin{equation*}\label{eq:elempredict}
 \Phi(t,\omega) = \sum_{k=1}^K\sum_{m=1}^M \mathbf 1_{(t_{k-1},t_k]\times B_{mk}}(t,\omega)
\sum_{n=1}^N h_n \otimes x_{kmn},\;\;\; t\geq 0, \omega \in
\Omega.
\end{equation*}
Here $0 \leq t_0 < \ldots < t_K <\infty$ is a finite increasing
sequence of nonegative numbers, the sets $B_{1k},\ldots,B_{Mk}$
belong to $\mathcal F_{t_{k-1}}$ for each  $k = 1,\,2,\,\ldots,
K$, and the vectors $h_1,\ldots,h_N$ are assumed to be orthogonal.
Suppose further that $M$ is an adapted local martingale taking
values in $H$. Then the {\em stochastic integral} $ \int \Phi \ud
M:\mathbb R_+ \times \Omega \to X$ of $\Phi$ with respect to $M$
is defined by the formula
\begin{equation*}
 \int_0^t\Phi \ud M = \sum_{k=1}^K\sum_{m=1}^M \mathbf 1_{B_{mk}}
\sum_{n=1}^N \langle(M(t_k\wedge t)- M(t_{k-1}\wedge t)),
h_n\rangle x_{kmn},\;\; t\geq 0.
\end{equation*}

\begin{remark}\label{rem:stochintgenPhi}
 If both $X$ and $H$ are finite dimensional, then we may assume that $X$ is isomorphic to $\mathbb R^d$, and thus by \cite[Theorem 26.6 and 26.12]{Kal} we can extend the stochastic integration from elementary predictable processes to all the predictable processes $\Phi:\mathbb R_+ \times \Omega \to \mathcal L(H,X)$ with 
 $$
 \mathbb E \Bigl( \sum_{i=1}^n\int_0^{\infty}\|\Phi h_i\|^2 \ud [\langle M, h_i\rangle]_s\Bigr)^{1/2}<\infty,
 $$ 
 where $n$ is the dimension of $H$ and $h_1,\ldots,h_n$ is an orthonormal basis of $H$. In fact, a similar characterization of stochastic integration can be shown for infinite dimensional $X$ and $H$ by using {\em $\gamma$-norms} (see \cite{NVW,Y18BDG,VY2016,Ver}). 
\end{remark}

\subsection{Hilbert transform}\label{subsec:HT}

Let $X$ be a Banach space.
 The {\em Hilbert transform} $\mathcal H_{X}$ is a singular integral
operator that maps a step function $f:\mathbb R\to X$ to the
function
 \begin{equation*}\label{eq:defdefofRHT}
   (\mathcal H_X f)(t):= \frac{1}{\pi}{\textnormal{p.v.}}\int_{\mathbb R}\frac{f(s)}{t-s}\ud s,\;\;\; t\in \mathbb R.
 \end{equation*}
For any $1<p<\infty$ we denote the norm of $\mathcal H_{X}$ on
$L^p(\mathbb R; X)$ by $\hbar_{p, X}$. Note that due to
\cite{Burk83,Bour83} we have that $\hbar_{p, X} <\infty$ if and
only if $X$ is UMD. Moreover, due to \cite{Gar85,Bour83} we have
that for every $1<p<\infty$
 \begin{equation}\label{eq:sqrtbetaleqhleqbeta^2}
   \sqrt{\beta_{p,X}} \leq \hbar_{p, X} \leq \beta_{p, X}^2.
 \end{equation}
 
\begin{remark}
 Recently in \cite{OY18} it was shown that $\hbar_{p, X}$ is the smallest constant $\hbar$ such that there exists a {\em plurisubharmonic function} $U_{\mathcal H}:X+iX\to \mathbb R$ (i.e.\ $z\mapsto U_{\mathcal H}(x_0 + iy_0 + z(x+iy))$ is subharmonic in $z\in \mathbb C$ for any fixed $x_0, y_0, x, y\in X$) such that $U_{\mathcal H}(x)\geq 0$ for any $x\in X$ and $U_{\mathcal H}(x+iy) \leq \hbar^p \|x\|^p - \|y\|^p$ for all $x, y\in X$.
\end{remark}

\subsection{Bellman functions and function approximation}\label{subsec:funcapprox}

Let $X$ be a UMD Banach space, $1<p<\infty$. Throughout the paper
we will use different {\em Bellman functions}, i.e.\ functions
$u:X \times X \to \mathbb R$ which have certain appropriate
properties. Let us outline which functions we will use
\begin{itemize}
\item the Burkholder function $U:X \times X \to\mathbb R$ (see
e.g.\ \cite{HNVW1} and the proof of Corollary
\ref{cor:estimforchipX}),

\item a plurisubharmonic function $U_{\mathcal H}:X+iX \to \mathbb
R$ (see \cite{OY18} and Subsection~\ref{subsec:HT}),

\item a diagonally plurisubharmonic function $U^{SO}:X+iX \to
\mathbb R$ (see Section~\ref{sec:chipX}).
\end{itemize}

For all the Bellman functions named above we may assume that $X$
is finite dimensional and that the function is twice Fr\'echet
differentiable by an approximation argument exploited in
\cite{Y17MartDec,OY18,BO12}. We will not repeat this argument
here, but just shortly remind the reader the main steps.
\begin{itemize}
\item Since $X$ is UMD, it is reflexive, and by the Pettis
measurability theorem \cite[Theorem 1.1.20]{HNVW1} we may assume
that $X$ is separable. Thus $X^*$ is separable as well, and there
exist an increasing sequence $(Y_n)_{n\geq 1}$ of finite
dimensional subspaces of $X^*$ such that $X^* = \overline{\cup_n
Y_n}$. Let $P_n:Y_n \to X^*$ be the injection operator. In the
sequel we will need to show that $\mathbb E \|\eta\|^p \leq c_{p,
X}^p  \mathbb E \|\xi\|^p $ for a certain pair of random variables $\xi, \eta \in
L^p(\Omega; X)$ and a certain constant $c_{p,
X}$. Since $\|P_n^* x\| \nearrow \|x\|$ monotonically as $n\to
\infty$ for any $x\in X$, by the monotone convergence theorem it
is sufficient to show that $\mathbb E \|P_n^*\eta\|^p \leq c_{p,
X}^p  \mathbb E \|P_n^*\xi\|^p $ for any $n\geq 1$. Moreover, in fact we need to show that $\mathbb E \|P_n^*\eta\|^p \leq c_{p,
Y_n^*}^p  \mathbb E \|P_n^*\xi\|^p $  since in our case $c_{p, X}$ equals either $\beta_{p, X}$, $\hbar_{p, X}$, or $\chi_{p, X}$ (see Section~\ref{sec:chipX} for the definition), and since all these constants can be represented as norms of operators having the same operators as their duals, so one has that analogously to \cite[Proposition 4.2.17]{HNVW1} $c_{p, X} = c_{p', X^*}$ (where $p' = p/(p-1)$), and in particular 
\[
 c_{p, Y_n^*} = c_{p', Y_n} \leq  c_{p', X^*} = c_{p, X},
\]
 Thus it is
sufficient to assume that $X$ is finite dimensional since both $P_n^* \xi$ and $P_n^* \eta$ have their values in a finite dimensional space $Y_n^*$.

\item Since $X$ is finite dimensional, for a Bellman function $u$ and
for any $\eps>0$ we can define $u_{\eps} := u *
\eps^{-1}\phi(\eps^{-1} \cdot)$, where $\phi:X \times X \to
\mathbb R_+$ is a $C^{\infty}$ function with a compact domain such
that $\int_{X \times X} \phi(x,y)\ud \lambda(x)\ud \lambda(y) = 1$
(here $\lambda$ is the {\em Lebesque measure} on $X$, see e.g.\
\cite[Remark 3.13]{Y17FourUMD} for the definition). Then
$u_{\eps}$ preserves such properties of $u$ as convexity,
concavity, or subharmonicity on a linear subspace of $X \times X$,
and $u_{\eps} \to u$ as $\eps \to 0$ locally uniformly on $X
\times X$ due to continuity of $u$. Therefore by this approximation argument we may assume that $u$ is $C^{\infty}$.
\end{itemize}

\section{The $\chi_{p,X}$ constant}\label{sec:chipX}

Let $X$ be a Banach space, $1<p<\infty$. We define $\chi_{p,X}\in
[0,\infty]$ to be the least number $\chi>0$ such that for any independent
standard Brownian motions $W, \widetilde W:\mathbb R_+ \times
\Omega \to \mathbb R$ and for any elementary predictable with
respect to the filtration generated by both $W$ and $\widetilde W$
process $\Phi:\mathbb R_+ \times \Omega \to X$ one has that
\[
\mathbb E \Bigl\| \int_0^{\infty} \Phi \ud \widetilde W \Bigr\|^p
\leq \chi^p \mathbb E \Bigl\| \int_0^{\infty} \Phi \ud W
\Bigr\|^p.
\]

\begin{remark}
$\chi_{p,X}$  can be equivalently defined in the following way. Let $(\gamma_n)_{n\geq 1}$ and $(\tilde \gamma_n)_{n\geq 1}$ be sequences of independent standard Gaussian random variables,  $\mathcal F_0 = \{\varnothing, \Omega\}$, and $\mathcal F_n = \sigma(\gamma_1, \tilde \gamma_1,\ldots,\gamma_n, \tilde \gamma_n)$ for $n\geq 1$. Then $\chi_{p, X}$ is the smallest $\chi>0$ such that for any $N\geq 1$ and any elementary step functions $v_0,\ldots,v_{N-1}:\Omega \to X$ with $v_n$ being $\mathcal F_n$-measurable for each $n =0, \ldots, N-1$, one has that
\begin{equation}\label{eq:defofchibygaussians}
\mathbb E\Bigl\|\sum_{n=1}^N\tilde\gamma_n v_{n-1}\Bigr\|^p \leq \chi^p  \mathbb E\Bigl\|\sum_{n=1}^N\gamma_n v_{n-1}\Bigr\|^p.
\end{equation}
Indeed, one can represent the sums $\sum_{n=1}^N\gamma_n v_{n-1}$ and $\sum_{n=1}^N\tilde \gamma_n v_{n-1}$  as stochastic integrals with respect to independent Brownian motions $W$ and $\widetilde W$ by just letting $\gamma_n = W_n-W_{n-1}$ and $\tilde \gamma_n = \widetilde W_n-\widetilde W_{n-1}$. On the other hand, if $W$ and $\widetilde W$ are independent Brownian motions and if $\Phi$ is elementary predictable and defined by 
\begin{equation*}
 \Phi(t,\omega) = \sum_{k=1}^K\sum_{m=1}^M \mathbf 1_{(t_{k-1},t_k]\times B_{mk}}(t,\omega) x_{km},\;\;\; t\geq 0, \omega \in
\Omega.
\end{equation*}
where $0 \leq t_0 < \ldots < t_K <\infty$ is a finite increasing
sequence of nonnegative numbers and the sets $B_{1k},\ldots,B_{Mk}$
belong to $\mathcal F_{t_{k-1}}$ for each  $k = 1,\,2,\,\ldots,
K$, then one can represent the stochastic integrals $\Phi \cdot W$ and $\Phi \cdot \widetilde W$ as the sums $\sum_{n=1}^N\gamma_n v_{n-1}$ and $\sum_{n=1}^N\tilde \gamma_n v_{n-1}$ in the following way
\[
 \int_0^{\infty} \Phi \ud W = \sum_{k=1}^K\sum_{m=1}^M \mathbf 1_{B_{mk}}
 (W(t_k)- W(t_{k-1})) x_{km} =  \sum_{k=1}^K v_{k-1}
\gamma_k,
\]
\[
 \int_0^{\infty} \Phi \ud \widetilde W = \sum_{k=1}^K\sum_{m=1}^M \mathbf 1_{B_{mk}}
 (\widetilde W(t_k)-\widetilde W(t_{k-1})) x_{km} =  \sum_{k=1}^K v_{k-1}
\tilde\gamma_k,
\]
where $\gamma_k =  \tfrac{W(t_k)- W(t_{k-1})}{\sqrt{t_k-t_{k-1}}}$, $\tilde\gamma_k =  \tfrac{\widetilde W(t_k)-\widetilde W(t_{k-1})}{\sqrt{t_k-t_{k-1}}}$, and 
$v_{k-1} =\sqrt{t_k-t_{k-1}} \sum_{m=1}^M \mathbf 1_{B_{mk}} x_{km}.$

The martingale transform \eqref{eq:defofchibygaussians} appears while working with Volterra-type operators and stochastic shifts (see \cite{GY19}).
\end{remark}

Concerning the constant $\chi_{p,X}$ one can show the following
proposition. First we will define diagonally plurisubharmonic functions.

\begin{definition}
A function $F:X + iX \to \mathbb R$ is called {\em diagonally
plurisubharmonic} if $z \mapsto F(x_0 + iy_0 + zx)$ is subharmonic
in $z\in \mathbb C$ for any $x_0, y_0, x\in X$.
\end{definition}

\begin{proposition}\label{prop:chi<inftyiffexistsdiagplsfunc}
Let $X$ be a Banach space, $1<p<\infty$. Then the following are
equivalent
\begin{enumerate}[(i)]
\item $\chi_{p, X}<\infty$,

\item there exists a constant $\chi >0$ and a
diagonally plurisubharmonic $u:X+ iX \to \mathbb R$ such that
$u(x)\geq 0$ for any $x\in X$, $x\mapsto u(x + iy)$ is convex in
$x\in X$ for any $y\in X$, $y\mapsto u(x + iy)$ is concave in
$y\in X$ for any $x\in X$, and
\begin{equation}\label{eq:ubiggethenchipxp-yp}
u(x + iy) \leq \chi^p \|x\|^p - \|y\|^p,\;\;\; x,\; y\in X.
\end{equation}
\end{enumerate}
Moreover, if this is the case, then the smallest $\chi$ for which
such a function $u$ exists equals $\chi_{p, X}$.
\end{proposition}

\begin{proof}
We will prove both implications separately.

{\em $(i) \Rightarrow (ii)$.} In order to show this implication we
need to construct function $u$ for $\chi = \chi_{p, X}$. In this
case let us define the desired function $u$ to be as follows
\begin{equation}\label{eq:defofdiagplsu}
\begin{split}
u(x + iy) := \inf\Big\{&\chi_{p, X}^p \mathbb E\Bigl\|x +
\int_0^{\infty} \Phi \ud W \Bigr\|^p - \mathbb E\Bigl\|y +
\int_0^{\infty} \Phi \ud \widetilde W \Bigr\|^p:\\
& \Phi:\mathbb R_+ \times \Omega \to X \textnormal{elementary
predictable}\Big\},\;\;\;x, y\in X.
\end{split}
\end{equation}
First of all notice that $u$ is finite on $X+iX$. Indeed, one has
that for any elementary predictable $\Phi:\mathbb R_+ \times
\Omega \to X$ and for any $x, y\in X$ by the triangle inequality
\begin{align*}
\chi_{p, X}^p \mathbb E\Bigl\|x &+ \int_0^{\infty} \Phi \ud W
\Bigr\|^p - \mathbb E\Bigl\|y + \int_0^{\infty} \Phi \ud
\widetilde W
\Bigr\|^p\\
& \gtrsim_{p} \chi_{p, X}^p \mathbb E\Bigl\| \int_0^{\infty} \Phi
\ud W \Bigr\|^p - \mathbb E\Bigl\| \int_0^{\infty} \Phi \ud
\widetilde W
\Bigr\|^p -\chi_{p, X}^p\|x\|^p - \|y\|^p\geq -\chi_{p, X}^p\|x\|^p - \|y\|^p,
\end{align*}
where the latter holds by the definition of $\chi_{p, X}$.

Let us show that $u$ is continuous. For any $x, y, \tilde x,
\tilde y$ one has that by the triangle inequality
\begin{equation*}
\begin{split}
u(x+iy) =\inf\Big\{&\chi_{p, X}^p \mathbb E\Bigl\|x +
\int_0^{\infty} \Phi \ud W \Bigr\|^p - \mathbb E\Bigl\|y +
\int_0^{\infty} \Phi \ud \widetilde W \Bigr\|^p:\\
& \Phi:\mathbb R_+ \times \Omega \to X \textnormal{elementary
predictable}\Big\}\\
\lesssim_{p} \inf\Big\{&\chi_{p, X}^p \mathbb E\Bigl\|\tilde x +
\int_0^{\infty} \Phi \ud W \Bigr\|^p - \mathbb E\Bigl\| \tilde y +
\int_0^{\infty} \Phi \ud \widetilde W \Bigr\|^p:\\
& \Phi:\mathbb R_+ \times \Omega \to X \textnormal{elementary
predictable}\Big\} + \chi_{p, X}\|x-\tilde x\|^p + \|y-\tilde y\|^p\\
\leq u(\tilde x + &i\tilde y) + \chi_{p, X}\|x-\tilde x\|^p + \|y-\tilde
y\|^p,
\end{split}
\end{equation*}
so the continuity follows.

Now let us show that $u$ is diagonally plurisubharmonic. Fix $x_0,
y_0, x\in X$. We need to show that $z \mapsto u(x_0 + iy_0 + zx)$
is subharmonic in $z\in \mathbb C$. To this end we need to
prove that for any fixed $r>0$
\begin{equation}\label{eq:diagplshofu}
u(x_0 + iy_0) \leq \frac {1}{2\pi}\int_{0}^{2\pi} u(x_0 + iy_0 +
xre^{i\theta}) \ud \theta.
\end{equation}
Let $W, \widetilde W:\mathbb R_+ \times \Omega \to \mathbb R$ be
independent standard Brownian motions. Define a stopping time
$\tau$ in the following way
\begin{equation*}\label{eq:defoftauoutofcircle}
\tau := \inf \{t\geq 0: W_t^2 + \widetilde W_t^2 = r\}.
\end{equation*}
Fix $\eps>0$. Note that since $u$ is continuous, there exist $\delta>0$
and a $\delta$-net $(a_n)_{n=1}^N = (x_n +
iy_n)_{n=1}^N$ of a compact set $A :=\{x_0 + iy_0 + xre^{i\theta}:
\theta \in [0,2\pi)\}\subset X + iX$ with
\begin{equation}\label{eq:controfanandepsassump}
|u(a) - u(a_n)| \leq \eps \;\; \forall a\in A \;\; \text{such
that}\;\|a-a_n\| <\delta
\end{equation}
(here the norm on $A$ is assumed to be a usual norm on
$\mathbb C$ since $A$ can be represented as a circle on $\mathbb
C$). Let $B_t := W_{t+\tau} -W_{\tau}$, $\widetilde B_t :=
\widetilde W_{t+\tau} -\widetilde W_{\tau}$. Note that $B$ and
$\widetilde B$ are independent Brownian motions (see e.g.\ \cite[Theorem 13.11]{Kal}).
Therefore by the definition of $u$ for every $n=1,\ldots,N$ there
exists an elementary predictable with respect to the filtration
generated by $B$ and $\widetilde B$ process $\Phi_n:\mathbb R_+
\time \Omega \to X$ such that
\begin{equation}\label{eq:sharpguysforu(an)}
u(a_n) \geq \chi_{p, X}^p \mathbb E\Bigl\|x_n + \int_0^{\infty}
\Phi_n \ud B \Bigr\|^p - \mathbb E\Bigl\|y_n + \int_0^{\infty}
\Phi_n \ud \widetilde B \Bigr\|^p - \eps.
\end{equation}
Now let us define a predictable with respect to the filtration
generated by $W$ and $\widetilde W$ process $\Phi$ in the
following way. $\Phi(t) = x$ if $t\leq \tau$ and $\Phi(t) =
\Phi_n(t-\tau)$ if $t> \tau$ and $a_n$ is the closest among the
set $(a_n)_{n=1}^N$ point to $x_0 + iy_0 + x(W_{\tau} +
i\widetilde W_{\tau})$. This is a predictable process
and since $\Phi$ takes values in a finite dimensional subspace of
$X$, it can be approximated by an elementary predictable process
(see Remark \ref{rem:stochintgenPhi}). Therefore we get that
\begin{align*}
u(x_0 + iy_0) &\leq \chi_{p, X}^p \mathbb E\Bigl\|x_0 +
\int_0^{\infty} \Phi \ud W \Bigr\|^p - \mathbb E\Bigl\|y_0 +
\int_0^{\infty} \Phi \ud \widetilde W \Bigr\|^p \\
&=  \chi_{p, X}^p \mathbb E\Bigl\|x_0 + xW_{\tau} +
\int_0^{\infty} \Phi(t) \ud B_{t-\tau} \Bigr\|^p \\
&\quad\quad - \mathbb E\Bigl\|y_0 + x\widetilde W_{\tau} +
\int_0^{\infty} \Phi(t) \ud \widetilde
B_{t-\tau} \Bigr\|^p\\
&\stackrel{(i)}=\frac {1}{2\pi}\int_{0}^{2\pi}   \chi_{p, X}^p
\mathbb E\Bigl\|x_0 + x \cos{\theta} + \int_0^{\infty}
\Phi_{n(\theta)}(t) \ud B_{t}
\Bigr\|^p\\
&\quad\quad - \mathbb E\Bigl\|y_0 + x \sin{\theta} +
\int_0^{\infty} \Phi_{n(\theta)}(t) \ud \widetilde B_{t} \Bigr\|^p
\ud
\theta\\
&\stackrel{(ii)}\leq \frac {1}{2\pi}\int_{0}^{2\pi}   \chi_{p,
X}^p \mathbb E\Bigl\|x_{n(\theta)} + \int_0^{\infty}
\Phi_{n(\theta)}(t) \ud B_{t}
\Bigr\|^p\\
&\quad\quad - \mathbb E\Bigl\|y_{n(\theta)} + \int_0^{\infty}
\Phi_{n(\theta)}(t) \ud \widetilde B_{t} \Bigr\|^p \ud
\theta + c_p \delta\\
&\stackrel{(iii)}\leq \frac {1}{2\pi}\int_{0}^{2\pi}
u(a_{n(\theta)}) + \eps \ud \theta +
c_p \delta\\
 &\stackrel{(iv)}\leq \frac {1}{2\pi}\int_{0}^{2\pi} u(x_0 + iy_0 +
xre^{i\theta}) \ud \theta + c_p \delta + 2\eps,
 \end{align*}
where $n(\theta)$ is such $n$ that $a_{n}$ is the closest to $x_0
+ iy_0 + xre^{i\theta}$ among $(a_n)_{n=1}^N$, $(i)$ follows from
the definition of $\Phi$, $(ii)$ holds by the triangle inequality
and the fact that $(a_n)_{n=1}^N$ is a $\delta$-net of $A$ (where
the constant $c_p$ depends only on $p$), $(iii)$ holds by
\eqref{eq:sharpguysforu(an)}, and $(iv)$ holds by
\eqref{eq:controfanandepsassump}. Now if $\eps\to 0$,
$\delta$ vanishes as well, and \eqref{eq:diagplshofu} follows.

Let us now show that $u(x)\geq 0$ for any $x\in X$. First notice
that $u$ is concave in the complex variable, i.e.\ $y\mapsto u(x +
iy)$ is concave in $y\in X$ for any $x\in X$, which follows
directly form the construction of $u$ in \eqref{eq:defofdiagplsu}.
Now one can show  that $u$ is convex in the real variable, i.e.\
$x\mapsto u(x + iy)$ is convex in $x\in X$ for any $y\in X$, by
using the same argument as was used for plurisubharmonic functions
in \cite[Subsection 2.6]{OY18}. Next notice that $u$ is symmetric,
i.e.\ $u(x+iy) = u(-x-iy)$ for any $x, y\in X$. Thus $x\mapsto
u(x)$ is a symmetric convex function with $u(0)=0$, so it is
nonnegative.

{\em $(ii) \Rightarrow (i)$.} Let $u: X +iX \to \mathbb R$ be a
function from $(ii)$. We need to show that  for any standard
Brownian motions $W, \widetilde W:\mathbb R_+ \times \Omega \to
\mathbb R$ and for any elementary predictable with respect to the
filtration generated by both $W$ and $\widetilde W$ process
$\Phi:\mathbb R_+ \times \Omega \to X$ one has that
\begin{equation}\label{eq:proofthatchifollfromdiagpls}
\mathbb E \Bigl\| \int_0^{\infty} \Phi \ud \widetilde W \Bigr\|^p
\leq \chi^p \mathbb E \Bigl\| \int_0^{\infty} \Phi \ud W
\Bigr\|^p.
\end{equation}
Since $\Phi$ is elementary predictable, it takes values in a
finite-dimensional subspace of $X$, so we may assume that $X$ is
finite-dimensional. Then by Subsection \ref{subsec:funcapprox} we can assume
that $u$ is twice differentiable on $X+ iX$ by a simple
convolution-type argument. Let $d<\infty$ be the dimension of $X$,
$(x_n)_{n=1}^d$ be the basis of $X$, $(x_n^*)_{n=1}^d$ be the {\em
corresponding dual basis} of $X^*$, i.e.\ a unique basis such that
$\langle x_n, x_m^*\rangle = \delta_{nm}$ for any $n,m=1,\ldots,d$
(see e.g.\ \cite{OY18,Y17MartDec,Y17FourUMD}). Then by It\^o's
formula \cite[Theorem 3.8]{Y17MartDec} and due to local
boundedness and twice differentiability of $u$ we have that (here
we define $M_t:=\int_0^t \Phi \ud W$ and $N_t:=\int_0^t \Phi \ud
\widetilde W$ for the convenience of the reader)
\begin{align}\label{eq:proofthatchiusingito}
 \chi^p \mathbb E \Bigl\| \int_0^{\infty} \Phi \ud W \Bigr\|^p
 - \mathbb E \Bigl\| \int_0^{\infty} \Phi \ud \widetilde W
 \Bigr\|^p &\geq \mathbb E u\Bigl( \int_0^{\infty} \Phi \ud W
 +i \int_0^{\infty} \Phi \ud \widetilde W \Bigr)\nonumber\\
 & =  \mathbb E u(M_0 + iN_0) + \mathbb E \int_0^{\infty} \bigl \langle \partial_x u( M_{t-} + iN_t), \ud
 M_t\bigr\rangle\\
 &\quad + \mathbb E  \int_0^{\infty} \bigl\langle \partial_{ix} u( M_{t-} + iN_t),\ud N_t\bigr\rangle + \frac 12\mathbb E I,\nonumber
\end{align}
where
\[
I =  \mathbb E \int_0^{\infty} \sum_{n,m=1}^d \Bigl(
\tfrac{\partial^2 u( M_{t-} + iN_t)}{\partial x_n x_m} +
 \tfrac{\partial^2 u( M_{t-} + iN_t)}{\partial ix_n
ix_m}\Bigr)
 \langle \Phi, x_n^*\rangle \cdot \langle \Phi, x_m^*\rangle \ud
 t.
\]
First notice that $\mathbb E u(M_0 + iN_0) = \mathbb E u(0) = 0$
and analogously to \cite[proof of Theorem 3.18]{Y17FourUMD} both $\partial_x u( M_{t-} + iN_t)$ and $\partial_{ix} u( M_{t-} + iN_t)$ are stochastically integrable with respect to $M$ and $N$ respectively, so
$$
\mathbb E \int_0^{\infty} \bigl \langle \partial_x u( M_{t-} + iN_t), \ud
 M_t\bigr\rangle + \mathbb E  \int_0^{\infty} \bigl\langle \partial_{ix} u( M_{t-} + iN_t),\ud N_t\bigr\rangle= 0,
$$
where the latter holds since both stochastic integrals are
martingales which start in zero. Let us show that $\mathbb E I \geq 0$. Fix
$t\geq 0$ and $\omega \in \Omega$. By \cite[Lemma 3.7]{Y17MartDec}
we are free to choose any basis (and the corresponding dual
basis). In particular, we can assume that $x_1 = \Phi(t, \omega)$.
Then $\langle \Phi(t, \omega), x_n^*\rangle = \delta_{1n}$ for any $1\leq n\leq d$, so
(here we skip $(t, \omega)$ for the convenience of the reader)
\begin{multline*}
\sum_{n,m=1}^d \Bigl( \tfrac{\partial^2 u( M_{t-} +
iN_t)}{\partial x_n x_m} +
 \tfrac{\partial^2 u( M_{t-} + iN_t)}{\partial ix_n
ix_m}\Bigr)
 \langle \Phi, x_n^*\rangle \cdot \langle \Phi, x_m^*\rangle\\
   = \tfrac{\partial^2 u( M_{t-} +
iN_t)}{\partial x_1^2} +
 \tfrac{\partial^2 u( M_{t-} + iN_t)}{\partial ix_1^2} = \Delta u(M_{t-} + iN_t + z x_1)\big|_{z=0} \geq 0,
\end{multline*}
where $z\in \mathbb C$, and the latter inequality follows from the
diagonal plurisubharmonicity of $u$. Thus $\mathbb E I \geq 0$, and hence
\eqref{eq:proofthatchifollfromdiagpls} follows from
\eqref{eq:proofthatchiusingito}.
\end{proof}

\begin{remark}\label{rem:defofUSO}
Note that the maximum of any set of harmonic functions is harmonic
as well, so the maximum of any set of diagonally plurisubharmonic functions is
diagonally plurisubharmonic as well, and thus for any Banach space
$X$ and for any $1<p<\infty$ with $\chi_{p, X}<\infty$ we can
define an {\em optimal}  diagonal plurisubharmonic function
$U^{SO}:X + iX \to \mathbb R$ as a supremum of all functions $u$
satisfying the conditions of Proposition
\ref{prop:chi<inftyiffexistsdiagplsfunc}$(ii)$.

Note that $U^{SO}$ coincides with the function $u$ defined by \eqref{eq:defofdiagplsu}. Indeed, let $U^{SO}$ be as defined above, $u$ be as in \eqref{eq:defofdiagplsu}. Then $U^{SO} \geq u$ by the definition of $U^{SO}$. Let us show that $U^{SO}(x+iy) \leq u(x+iy)$ for any $x, y\in X$. First fix independent Brownian motions $W$ and $\widetilde W$ and elementary predictable $\Phi:\mathbb R_+ \times \Omega \to X$. Then similarly to the It\^o argument from the proof of Proposition \ref{prop:chi<inftyiffexistsdiagplsfunc} one has that
\[
 U^{SO} (x+iy) \leq \mathbb E U\Bigl(x+iy + \int_0^{\infty} \Phi \ud W + i  \int_0^{\infty} \Phi \ud \widetilde W\Bigr).
\]
Thus
\begin{align*}
 U^{SO} (x+iy)& \leq \inf\Bigl\{\mathbb E U\Bigl(x+iy + \int_0^{\infty} \Phi \ud W + i  \int_0^{\infty} \Phi \ud \widetilde W\Bigr):\\
 &\quad\quad\Phi \;\text{elementary predictable}\Bigr\} \leq u(x+iy),
\end{align*}
which implies the desired.
\end{remark}

As a corollary of Proposition \ref{prop:chi<inftyiffexistsdiagplsfunc} one can show the following upper and lower bounds for $\chi_{p, X}$. Recall that we define {\em decoupling constants} $\beta_{p,X}^{\gamma ,+}$ and $\beta_{p,X}^{\gamma ,-}$ to be the smallest possible $\beta^+$ and $\beta^-$ respectively for which
 \[
  \frac{1}{(\beta^-)^p}\mathbb E\Bigl\|\int_0^\infty \Phi \ud W\Bigr\|^p \leq \mathbb E\Bigl\|\int_0^\infty \Phi \ud \widetilde W\Bigr\|^p  \leq (\beta^+)^p \mathbb E\Bigl\|\int_0^\infty \Phi \ud W\Bigr\|^p,
  \]
where $W$ and $\widetilde{W}$ are independent standard Brownian motion, $\Phi:\mathbb R_+ \times \Omega \to X$ is elementary predictable which is independent of $\widetilde{W}$ (we refer the reader to \cite{Gar85,HNVW1,Ver07,MC,Geiss99,CG,OY18} for further details on decoupling constants).

\begin{corollary}\label{cor:estimforchipX}
Let $X$ be a Banach space, $1<p<\infty$. Then $\chi_{p,X}<\infty$
if and only if $X$ is a UMD Banach space. Moreover, if this is the
case, then
\begin{equation}\label{eq:uppandlowboundforchipX}
\max\Bigl\{\sqrt{\beta_{p, X}}, \sqrt{\hbar_{p,X}}\Bigr\} \stackrel{(i)}\leq \max\{\beta_{p, X}^{\gamma,+}, \beta_{p, X}^{\gamma, -}\} \stackrel{(ii)}\leq \chi_{p, X} \stackrel{(iii)}\leq \min\{\beta_{p, X}, \hbar_{p,X}\}.
\end{equation}
\end{corollary}

\begin{proof}
First we show \eqref{eq:uppandlowboundforchipX}, and then the
``iff'' statement will follow simultaneously. Let first show
$(iii)$ in \eqref{eq:uppandlowboundforchipX}. The fact that
$\chi_{p, X} \leq  \hbar_{p,X}$ follows from \cite{OY18}, the
definition of $\chi_{p,X}$, and the fact that any two stochastic
integrals $\int \Phi \ud W$ and $\int \Phi \ud \widetilde W$ are
orthogonal martingales weakly differentially subordinate to each
other. The inequality $\chi_{p, X} \leq\beta_{p, X}$ can be proven
using a standard Burkholder function argument e.g.\ presented in
\cite{Y17FourUMD,Y17MartDec}. Indeed, if $\beta_{p, X}<\infty$,
then $X$ is a UMD Banach space, and their exists a {\em
zigzag-concave} function $U:X\times X \to \mathbb R$ (i.e.\ $z
\mapsto U(x+z, y+\alpha z)$ is concave in $z\in X$ for any $x,
y\in X$ and $\alpha \in [-1,1]$) such that $U(0,0)=0$ and
\[
U(x, y) \geq \|y\|^p - \beta_{p, X}^p\|x\|^p,\;\;\; x,y\in X.
\]
(This function is called {\em Burkholder}.) By a standard
convolution-type argument (see Subsection \ref{subsec:funcapprox}) we may
assume that $U$ is twice differentiable, and hence for any
independent standard Brownian motions $W$ and $\widetilde W$ and
for any elementary predictable $\Phi:\mathbb R_+ \times \Omega \to
X$ by It\^o's formula \cite[Theorem 3.8]{Y17MartDec} we have that
analogously to \eqref{eq:proofthatchiusingito} with denoting $M :=
\int \Phi \ud W$ and $ N := \int \Phi \ud \widetilde W$
\begin{align*}
\mathbb E \Bigl\| \int_0^{\infty} \Phi \ud \widetilde W \Bigr\|^p
- \beta_{p,X}^p \mathbb E \Bigl\| \int_0^{\infty} \Phi \ud W
\Bigr\|^p &\leq U\Bigl( \int_0^{\infty} \Phi \ud W , \int_0^{\infty} \Phi \ud \widetilde W\Bigr)\\
& = \frac12\int_0^{\infty} \tfrac{\partial^2 U(M_t, N_t)}{\partial (\Phi, 0)^2} + \tfrac{\partial^2 U(M_t, N_t)}{\partial (0,\Phi)^2} \ud t \\
&= \frac14\int_0^{\infty} \tfrac{\partial^2 U(M_t, N_t)}{\partial
(\Phi, \Phi)^2} + \tfrac{\partial^2 U(M_t, N_t)}{\partial
(\Phi,-\Phi)^2} \ud t \leq 0,
\end{align*}
where the latter inequality holds due to the zigzag-concavity of $U$ (so both $\tfrac{\partial^2 U(x, y)}{\partial (z, z)^2}$ and $\tfrac{\partial^2 U(x, y)}{\partial (z, -z)^2}$ and nonnegative for any $x,y, z\in X$). Thus $\chi_{p, X} \leq\beta_{p, X}$ holds true.

Now $(ii)$ of \eqref{eq:uppandlowboundforchipX} follows directly from the definitions of $\chi_{p, X}$, $\beta_{p, X}^{\gamma,+}$, and $\beta_{p, X}^{\gamma,-}$, while $(i)$ holds by \cite[p.\ 43 and Theorem 3]{Gar85}.
\end{proof}

\begin{remark}\label{rem:-UandUHarediagplsh}
Note that due to the latter proof for a Burkholder function $U$
one has that $-U$ is diagonal plurisubharmonic. Thus the proof of
$(iii)$ of \eqref{eq:uppandlowboundforchipX} has the following
form: {\em both $-U$ and $U_{\mathcal H}$ are diagonally
plurisubharmonic and thus satisfy the conditions of Proposition
\ref{prop:chi<inftyiffexistsdiagplsfunc}$(ii)$}, so the upper bound $(iii)$ of \eqref{eq:uppandlowboundforchipX} holds true. 

We wish to notice that in the real-valued case functions $U^{SO}$ and $U_{\mathcal H}$ coincide since in this case there is no difference between plurisubharmonicity and diagonal  plurisubharmonicity. Nevertheless, if the same holds for a general UMD Banach space, then $\hbar_{p, X} = \chi_{p, X}\leq \beta_{p, X}$, which would partly solve an open problem outlined in the introduction.
\end{remark}

\section{Weak differential subordination\\ of strongly orthogonal martingales}\label{sec:WDSofSOM}

Now we are ready to show the main result of the paper.

\begin{theorem}\label{thm:WDSforSOM}
Let $X$ be a UMD Banach space, $1<p<\infty$. Then for any strongly
orthogonal martingales $M,N:\mathbb R_+ \times \Omega \to X$ with
$N\stackrel{w} {\ll} M$ one has that
\[
\mathbb E \|N_t\|^p \leq \chi_{p, X}^p \mathbb E \|M_t\|^p, \;\;\;
t\geq 0.
\]
\end{theorem}

\begin{proof}
By Subsection \ref{subsec:funcapprox} we may assume that $X$ is
finite dimensional and that all the Bellman functions are smooth.
Due to \eqref{eq:ubiggethenchipxp-yp} we only need to show that
\begin{equation}\label{eq:proofofmainthmBellmanuse}
\mathbb E U^{SO}(M_t + iN_t) \geq 0,
\end{equation}
where $U^{SO}$ is as in Remark \ref{rem:defofUSO}. Let
$d\geq 0$ be the dimension of $X$. Since $N\stackrel{w} {\ll} M$
and since $M$ and $N$ are orthogonal, by \cite[Section 3]{OY18} we know
that after a proper time-change there exist a standard
$2d$-dimensional Brownian motion $W$ and predictable $\Phi,
\Psi:\mathbb R_+ \times \Omega \to \mathcal L(\mathbb R^{2d}, X)$
which are stochastically integrable with respect to $W$ such that
$N = \int \Psi \ud W$ and $M = M_0 + \int \Phi \ud W + M^d$, where
$M^d$ is purely discontinuous (see Subsection \ref{subsec:qvandpdm}). Moreover, as $M$ and $N$ are
strongly orthogonal, we have that for any $x^*, y^*\in X^*$ and
$t\geq 0$ by \cite[Theorem 26.6 and 26.13]{Kal}
\[
[\langle M, x^* \rangle,\langle N, y^* \rangle]_t = \int_0^t \bigl \langle \Phi^*(s)x^*, \Psi^*(s)y^*\bigr \rangle \ud s = 0.
\]
Therefore by the Lebesgue differentiation theorem $\langle \Phi^*x^*, \Psi^*y^*\bigr \rangle = 0$ a.e.\ on $\mathbb R_+ \times \Omega$. By choosing $(x^*, y^*)$ from a dense subset of $X^* \times X^*$ and using the fact that $(x^*, y^*) \mapsto \langle \Phi^*x^*, \Psi^*y^*\bigr \rangle$ is continuous on $X^* \times X^*$ on the whole $\mathbb R_+ \times \Omega$, one has
\begin{equation}\label{eq:Phix*Psix*=0}
\langle \Phi^*x^*, \Psi^*y^*\bigr \rangle = 0,\;\;\; x^*, y^* \in X^*,
\end{equation}
a.e.\ on $\mathbb R_+ \times \Omega$. Furthermore, by
\cite[Section 3]{OY18} we have that a.s.\ for any $0\leq s\leq t$ there
exists a skew-symmetric operator $A(s, \omega)\in \mathcal
L(\mathbb R^d)$ (i.e.\ $\langle Ah , h\rangle = 0$ for any $h\in
\mathbb R^d$) of norm at most one such that
\begin{equation}\label{eq:existofA}
\Psi(s, \omega)  = \Phi(s,\omega) A (s,\omega).
\end{equation}

Now let us show \eqref{eq:proofofmainthmBellmanuse} using \eqref{eq:Phix*Psix*=0}.  Let $(x_n)_{n=1}^d$ be a basis of $X$, $(x_n^*)_{n=1}^d$ be the corresponding dual basis of $X^*$. By It\^o's formula \cite[Theorem 3.8]{Y17MartDec} and smoothness of $U^{SO}$ we have that
\begin{align*}
\mathbb E U^{SO}(M_t + iN_t) = \mathbb E U^{SO}(M_0 + iN_0) + \mathbb E I_1 + \mathbb E I_2 + \frac 12 \mathbb E I_3,
\end{align*}
where
$$
I_1 = \int_0^t \langle \partial U^{SO} (M_{s-} + iN_s), \ud M_s + i\ud N_s \rangle,
$$
$$
I_2 = \sum_{0\leq s \leq t} \Delta U^{SO} (M_s + iN_s)  - \langle \partial U^{SO} (M_{s-} + iN_s), \Delta M_s \rangle,
$$
and
\begin{align*}
I_3 &= \int_0^t \sum_{n, m=1}^d\tfrac{\partial^2 U^{SO} (M_{s-} + iN_s)}{\partial x_nx_m} \langle \Phi^*x_n^*,\Phi^* x_m^* \rangle\ud t\\
&\quad + 2\int_0^t \sum_{n, m=1}^d\tfrac{\partial^2 U^{SO} (M_{s-} + iN_s)}{\partial x_n ix_m} \langle \Phi^*x_n^*,\Psi^* x_m^* \rangle\ud t\\
&\quad\quad + \int_0^t \sum_{n, m=1}^d\tfrac{\partial^2 U^{SO} (M_{s-} + iN_s)}{\partial ix_n ix_m} \langle \Psi^*x_n^*,\Psi^* x_m^* \rangle\ud t.
\end{align*}
First notice that since $N_0=0$ and since $U^{SO}(x) \geq 0$ for any $x\in X$ we have that $\mathbb E U^{SO}(M_0 + iN_0) = \mathbb E U^{SO}(M_0) \geq 0$. Moreover, $\mathbb E I_1 = 0$ since this is a martingale that starts at zero (which follows similarly to the proof of Proposition \ref{prop:chi<inftyiffexistsdiagplsfunc}). Let us show that $I_2 \geq 0$ a.s. Note that $x \mapsto U^{SO}(x + iy)$ is convex in $x\in X$ for any $y\in X$ by Proposition \ref{prop:chi<inftyiffexistsdiagplsfunc}, so by the continuity of $N$ we have that for any $0\leq s \leq t$
\begin{align*}
 U^{SO} (M_s + iN_s)  \leq  U^{SO} (M_{s-} + iN_s)  +  \langle \partial U^{SO} (M_{s-} + iN_s), \Delta M_s \rangle,
\end{align*}
and thus $I_2 \geq 0$ a.s. 

Now we show that $I_3 \geq 0$ a.s. In order
to show this we need to prove that a.s.\ for every $0\leq s\leq t$
\begin{equation}\label{I3pointwisegeq0}
\begin{split}
\sum_{n, m=1}^d&\tfrac{\partial^2 U^{SO} (M_{s-} + iN_s)}{\partial
x_nx_m} \langle \Phi^*x_n^*,\Phi^* x_m^* \rangle\\
& \quad + \tfrac{\partial^2 U^{SO} (M_{s-} + iN_s)}{\partial x_n
ix_m}
\langle \Phi^*x_n^*,\Psi^* x_m^* \rangle\\
 &\quad \quad+ \tfrac{\partial^2
U^{SO} (M_{s-} + iN_s)}{\partial ix_n ix_m} \langle
\Psi^*x_n^*,\Psi^* x_m^* \rangle \geq 0
\end{split}
\end{equation}
Fix $\omega \in \Omega$ and $0\leq s\leq t$ so that
\eqref{eq:Phix*Psix*=0} and \eqref{eq:existofA} hold true. Then the
expression on the left-hand side of \eqref{I3pointwisegeq0} gets
the following form
\begin{equation}\label{eq:I3simplerform}
\sum_{n, m=1}^d\tfrac{\partial^2 U^{SO} (M_{s-} + iN_s)}{\partial
x_nx_m} \langle \Phi^*x_n^*,\Phi^* x_m^* \rangle +
\tfrac{\partial^2 U^{SO} (M_{s-} + iN_s)}{\partial ix_n ix_m}
\langle \Psi^*x_n^*,\Psi^* x_m^* \rangle.
\end{equation}
Now analogously to \cite[Section 3]{OY18} the expression
\eqref{eq:I3simplerform} does not depend on the choice of the
basis $(x_n)_{n=1}^d$ or, equivalently, the choice of the basis
$(x_n^*)_{n=1}^d$ (since one can reconstruct the basis by its
corresponding dual basis, see \cite{OY18,Y17MartDec}). Moreover, by
\eqref{eq:existofA} for two symmetric nonnegative bilinear forms
$V, W :X^*\times X^* \to \mathbb R$ defined~by
\[
V(x^*, y^*) := \langle \Phi^*x^*,\Phi^* y^* \rangle,\;\;  W(x^*,
y^*) := \langle \Psi^*x^*,\Psi^* y^* \rangle,\;\;\;\;\; x^*,
y^*\in X^*,
\]
we have that $V(x^*, x^*) = 0$ implies $W(x^*, x^*) =0$ for any
$x^*\in X^*$. Thus  by \cite[Section 3]{OY18} there exist a basis
$(y_n^*)_{n=1}^d$ of $X^*$ with the corresponding dual basis
$(y_n)_{n=1}^d$ of $X$, a $[0,1]$-valued sequence
$(\lambda_n)_{n=1}^d$, and a number $0 \leq K \leq d$ such that
$V(y_n^*, y_m^*) = \delta_{nm} \mathbf 1_{m,n \leq K}$ and
$W(y_n^*, y_m^*) = \lambda_n\delta_{nm} \mathbf 1_{m,n \leq K}$
for any $m,n=1,\ldots,d$. Therefore by the discussion above we can
change the basis and get that the expression
\eqref{eq:I3simplerform} equals
\begin{equation}\label{eq:I3withnewbasis}
\begin{split}
\sum_{n, m=1}^d\tfrac{\partial^2 U^{SO} (M_{s-} + iN_s)}{\partial
y_ny_m}& \langle \Phi^*y_n^*,\Phi^* y_m^* \rangle +
\tfrac{\partial^2 U^{SO} (M_{s-} + iN_s)}{\partial iy_n iy_m}
\langle \Psi^*y_n^*,\Psi^* y_m^* \rangle\\
&= \sum_{n=1}^K\tfrac{\partial^2 U^{SO} (M_{s-} + iN_s)}{\partial
y_n^2}  + \lambda_n\tfrac{\partial^2 U^{SO} (M_{s-} +
iN_s)}{\partial iy_n^2}.
\end{split}
\end{equation}
Since $y\mapsto U^{SO}(x+iy)$ is concave in $y\in X$ for any $x\in
X$, $\tfrac{\partial^2 U^{SO} (M_{s-} + iN_s)}{\partial iy_n^2}
\leq 0$, and hence due to the fact that $ 0\leq \lambda_n \leq 1$
we have that the latter expression of \eqref{eq:I3withnewbasis} is
bounded from below by (here $z\in \mathbb C$)
\begin{align*}
\sum_{n=1}^K\tfrac{\partial^2 U^{SO} (M_{s-} + iN_s)}{\partial
y_n^2}  + \tfrac{\partial^2 U^{SO} (M_{s-} + iN_s)}{\partial iy_n^2} = \sum_{n=1}^K \Delta_z U^{SO} (M_{s-} + iN_s + zy_n)|_{z=0}
\geq 0,
\end{align*}
where the latter holds by the diagonal plurisubharmonicity of
$U^{SO}$. Therefore \eqref{I3pointwisegeq0} holds a.e.\ on
$\mathbb R_+ \times \Omega$, and thus $\mathbb EI_3 \geq 0$. This
completes the proof of \eqref{eq:proofofmainthmBellmanuse} and the
proof of the theorem.
\end{proof}

\bibliographystyle{plain}

\def\cprime{$'$} \def\polhk#1{\setbox0=\hbox{#1}{\ooalign{\hidewidth
  \lower1.5ex\hbox{`}\hidewidth\crcr\unhbox0}}}
  \def\polhk#1{\setbox0=\hbox{#1}{\ooalign{\hidewidth
  \lower1.5ex\hbox{`}\hidewidth\crcr\unhbox0}}} \def\cprime{$'$}

\end{document}